\documentclass[10pt,twoside]{siamltex}
\usepackage{amsfonts,epsfig,color}
\usepackage{color,soul}
\usepackage{amssymb,amsmath,bm,url}
\usepackage{color}
\usepackage[lined,boxruled]{algorithm2e}
\usepackage{comment}
\usepackage{calrsfs}
\usepackage{amsfonts}
\usepackage{mathrsfs}
\usepackage{rotating}
\usepackage{multirow}
\usepackage{multicol}
\usepackage{array}
\usepackage{booktabs}
\usepackage{mathptmx}
\usepackage{epsfig}
\usepackage{amsmath}
\usepackage[mathscr]{euscript}
\usepackage{color}
\usepackage{amssymb}
\usepackage{upgreek}
\usepackage{hhline,booktabs}
\usepackage{multirow}
\usepackage{multirow,bigstrut,rotating,float}
\usepackage{algorithmic}
\usepackage{ragged2e}
\usepackage{kantlipsum}
\usepackage{subfigure}
\usepackage{algorithm2e}
\newtheorem{remark}[theorem]{Remark}
\newtheorem{example}[theorem]{Example}
\def\tr{{\rm tr}}

\def\max{{\rm max}}
\def\min{{\rm min}}

\begin{document}
	
	\thispagestyle{empty}
	\bibliographystyle{siam}

	\title{On the Golub--Kahan bidiagonalization for ill-posed tensor equations with applications to color image restoration}
	\author{Fatemeh P. A. Beik$^1$, Khalide Jbilou$^2$, Mehdi Najafi-Kalyani$^{1}$ and Lothar Reichel$^3$}
	
	\footnotetext[1]{Department of Mathematics, Vali-e-Asr University of Rafsanjan, PO Box 518, Rafsanjan, Iran (f.beik@vru.ac.ir (F. P. A. Beik); mehdi.najafi@stu.vru.ac.ir (M. Najafi-Kalyani)).}
	\footnotetext[2]{Laboratory LMPA, 50 Rue F. Buisson, ULCO calais cedex, France
		(jbilou@univ-littoral.fr).
		}
	\footnotetext[3]{Department of Mathematical Sciences, Kent State University,
		Kent, OH 44242, USA	(reichel@math.kent.edu).
	}
	\maketitle
	
	\markboth{{\sc F. P.~A.~Beik, K.~Jbilou, M.~Najafi-Kalyani and L.~Reichel}}
	{\sc Golub--Kahan bidiagonalization for ill-posed tensor equations}

	\begin{abstract}
This paper is concerned with solving ill-posed tensor linear  equations. These kinds of equations may appear from finite difference discretization of high-dimensional convection-diffusion problems or when partial differential equations in many dimensions are discretized by collocation spectral methods. Here, we propose  the Tensor Golub--Kahan bidiagonalization (TGKB) algorithm  in conjunction with the well known Tikhonov regularization method  to solve the mentioned problems.
Theoretical results are presented to discuss on conditioning of the Stein  tensor equation and to reveal that how the TGKB process can be exploited for general tensor equations. In the last section, some classical test problems  are examined to numerically illustrate the feasibility of proposed algorithms and also applications for color image restoration are considered.
	\end{abstract}
	
	\begin{keywords}
Tensor linear operator equation, Ill-posed problem, Tikhonov regularization, Golub--Kahan bidiagonalization.
	\end{keywords}
	
	\begin{AMS}
		65F10, 15A24
	\end{AMS}

	\section{Introduction}
	\label{intro}
	
	This paper deals with  solving severely ill-conditioned tensor equations. We are particularly interested in  Sylvester  and Stein  tensor equations. It should be commented the proposed iterative schemes can be used for solving,
	\begin{equation}\label{eq131}
		\mathcal{L}(\mathscr{X})=\mathscr{C},
\end{equation}
	where $\mathcal{L}:\mathbb{R}^{I_1\times I_2\times \ldots \times I_N} \to \mathbb{R}^{I_1\times I_2\times \ldots \times I_N}$ is an arbitrary linear tensor operator. An ill-posed tensor equation may appear in color image restoration, video restoration, and when solving certain partial differential equations by collocation methods in several space dimensions \cite{Beik,Malek1,Malek2,Momeni-Masuleh,Najafi}. Throughout this work, vectors and matrices are respectively denoted by lowercase and capital letters, and  tensors of order three (or higher) are represented by Euler script letters. Before stating the main problems, we need to recall the definition of $n$-mode product from \cite{Kolda}.

	\begin{definition} \label{def1.1}  The $n$-mode (matrix) product of a tensor $\mathscr{X}\in \mathbb{R}^{I_1\times I_2\times \ldots \times I_N}$ with a matrix $U\in \mathbb{R}^{J\times I_n}$ is denoted by $\mathscr{X} \times_n U$ and is of size
		\[{{{I_1\times \cdots\times I_{n-1}\times J \times I_{n+1}\times\cdots \times I_N}}},\]
		and its elements are defined as follows:
		\[(\mathscr{X} \times_n U)_{_{i_1\cdots i_{n-1}ji_{n+1}\cdots i_N}} = \sum_{i_n=1}^{I_n} x_{i_1i_2\cdots i_N}u_{ji_n}.\]
	\end{definition}
	
	The Sylvester and Stein  tensor equations are respectively given by
	\begin{equation}\label{eq1}
	\mathscr{X}\times_1 A^{(1)}+\mathscr{X}\times_2 A^{(2)}+\ldots+\mathscr{X}\times_N
	A^{(N)}=\mathscr{D},
	\end{equation}
	and
	\begin{equation}\label{eql1}
	\mathscr{X}-\mathscr{X}\times_1 A^{(1)}\times_2 A^{(2)}\ldots\times_N
	A^{(N)}=\mathscr{F},
	\end{equation}
	where the right-hand side tensors $\mathscr{D}, \mathscr{F} \in \mathbb{R}^{I_1\times I_2 \times \ldots \times I_N}$ and the coefficient
	matrices $A^{(n)}\in \mathbb{R}^{I_n\times I_n}$  ($n=1,2,\ldots,N$) are known, and
	$\mathscr{X}\in \mathbb{R}^{I_1\times I_2 \times \ldots \times I_N}$  is the unknown
	tensor.

	Sylvester tensor equation may arise from the discretization of a linear partial differential equation in several space-dimensions by finite differences \cite{Ballani,Beik,Chen} or by spectral methods \cite{Beik,Malek1,Malek2,Momeni-Masuleh,Song}. Some discussions on conditioning of \eqref{eq1}  under certain conditions are presented in \cite{Najafi} where Najafi et al. proposed using the standard Tikhonov regularization technique in conjunction with global Hessenberg processes in tensor form to solve \eqref{eq1} with perturbed right-hand sides. Some results for perturbation analysis of \eqref{eql1} are given in  \cite{Liang} and  a more recent work by Xu and Wang \cite{Xu} where Eq. \eqref{eql1} is solved by tensor form  of the BiCG and BiCR methods.  Liang and Zheng \cite{Liang}  established some results for perturbation analysis of \eqref{eql1} in the case $N$ is even and $A^{(1)}=\cdots=A^{(N)}=A$ with $A$ being a Schur stable (all the eigenvalues of $A$ lie in the open unite disc). However, presented results rely on the matrix two norm of $( {I - {A^{(N)}} \otimes  \cdots  \otimes {A^{(2)}} \otimes {A^{(1)}}} )^{-1}.$
	
	More recently, Huang et al. \cite{Huang2019} proposed global form of well--known iterative methods in their tensor forms to solve a class of tensor equations via the Einstein product. Here, we comment that the proposed iterative approach in this work can
	be also used when the mentioned problem in \cite{Huang2019} is ill-posed.
	
	In this paper, we first establish some results to analyze the conditioning of \eqref{eql1} motivated by \cite{Liang, Xu}. Then the tensor form of the GKB process is proposed
	for solving ill-posed tensor equations. More precisely, we illustrate how
	 tensor--based GKB process can be exploited to solve ill-posed problems  \eqref{eq1} and \eqref{eql1}. To this end, we apply the established results in \cite{Beik} and generalize exploited
	 techniques of \cite{Bentbib}. It is immediate to observe that the results (in Section \ref{Sec3}) can be also used for solving ill-posed problem of the general form \eqref{eq131}.
	
	The remainder of paper is organized as follows. Before ending this section, we present some symbols and notations used throughout next sections. We further recall the concept of contract product between two tensors. In Section \ref{Sec2}, we present some results related to sensitivity analysis of \eqref{eql1}. Section \ref{Sec3} is devoted for constructing an approach based on tensor form of GKB and Gauss-type quadrature in conjunction with Tikhonov regularization technique to solve ill-posed  tensor equations.  In order to illustrate the effectiveness of proposed iterative schemes, some numerical results are reported in Section \ref{Sec4}.
	Finally the paper is ended with a brief conclusion in Section \ref{Sec5}.
	
	\subsection{Notations}
		Given a $N$-mode tensor $\mathscr{X}\in \mathbb{R}^{I_1\times I_2\times \cdots \times I_N}$, the notation $x_{i_1i_2\ldots i_N}$ stands for element $(i_1,i_2,\ldots,i_N)$ of $\mathscr{X}$.  For a given square matrix $A$ with real eigenvalues, we denote the minimum and maximum eigenvalues of $A$ by $\lambda_{\min}(A)$ and $\lambda_{\max}(A)$, respectively. The  set of all eigenvalues (spectrum) of $A$ is signified by $\sigma(A)$. The symmetric and skew-symmetric parts of $A$ are respectively denoted by $\mathcal{H}(A)$ and $\mathcal{S}(A)$, i.e.,
	\[\mathcal{H}(A)=\frac{1}{2}(A+A^T)\quad \text{and} \quad \mathcal{S}(A)=\frac{1}{2}(A-A^T).\]
	By condition number of an invertible matrix $A$, we mean ``$\mathrm{cond}(A)=\|A\|_2\|A^{-1}\|_2$"
	where $\|.\|_2$ is the matrix $2$-norm. The notation $\mathop  \bigotimes \limits_{i = 1}^\ell x_i:= x_1\otimes x_2\otimes \ldots \otimes x_\ell$ is exploited for multi-dimensional  Kronecker product. The vector $\textrm{vec}(\mathscr{X})$  is obtained by using the standard vectorization operator with respect to frontal slices of $\mathscr{X}$. The mode-$n$ matrization of a given tensor $\mathscr{X}$ is  denoted by $X_{(n)}$ which arranges
	the   mode-$n$ fibers to be the columns of resulting matrix. We recall that a fiber is defined by fixing every index but one; see \cite{Kolda} for more details.

	\subsection{Contracted product}
	The $\boxtimes^N$ product between two $N$-mode tensors
	\[
	\mathscr{X}\in \mathbb{R}^{I_1\times I_2 \times \cdots \times I_{N-1} \times  I_N} \quad \text{and} \quad  \mathscr{Y}\in \mathbb{R}^{I_1\times {I}_2 \times \cdots \times I_{N-1} \times \tilde{I}_N},
	\]
	 is defined as an $I_N \times \tilde{I}_N$ matrix whose $(i,j)$-th entry is
	\[
	[\mathscr{X} \boxtimes^N \mathscr{Y}]_{ij}=\text{tr} (\mathscr{X}_{{::\dots:}i} \boxtimes^{N-1} \mathscr{Y}_{{::\dots:}j}),\qquad N=3,4,\ldots,
	\]
	where
	\[\mathscr{X} \boxtimes^2 \mathscr{Y}= \mathscr{X}^T \mathscr{Y}, \qquad \mathscr{X}\in \mathbb{R}^{I_1\times I_2}, \mathscr{Y}\in \mathbb{R}^{I_1\times \tilde{I}_2}.\]
	The $\boxtimes^N$ product can be mentioned as a special case of the contracted product \cite{Cichocki}. More precisely, $\mathscr{X} \boxtimes^N \mathscr{Y}$ is the contracted product of $N$-mode tensors $\mathscr{X}$ and $\mathscr{Y}$ along the first $N-1$ modes.
	For $\mathscr{X}, \mathscr{Y} \in \mathbb{R}^{I_1\times I_2 \times \cdots \times I_N}$, it can be observed that
	\begin{equation}\label{20-1}
	\left\langle {\mathscr{X}, \mathscr{Y}} \right\rangle =\text{tr}(\mathscr{X} \boxtimes^N \mathscr{Y}),\qquad N=2,3,\ldots,
	\end{equation}
	and
	$\left\| \mathscr{X} \right\|^2= \tr (\mathscr{X} \boxtimes^N \mathscr{X})=\mathscr{X} \boxtimes^{(N+1)} \mathscr{X}$ for $\mathscr{X}\in \mathbb{R}^{I_1\times I_2 \times \cdots \times I_N}$.\\
	We finish this part, by recalling the following two useful results from \cite{Beik}.
	
	\begin{lemma}\label{lem3.1}
		If $\mathscr{X}\in \mathbb{R}^{I_1\times\cdots\times I_n\times \cdots \times I_N}$, $A\in\mathbb{R}^{J_n\times I_n}$ and $y\in\mathbb{R}^{J_n}$, then we have
		{\[\mathscr{X}\times_n A \bar{\times}_n y=\mathscr{X}\bar{\times}_n (A^Ty).\]}
	\end{lemma}
	
	\begin{proposition} \label{p2} Suppose that $\mathscr{B}\in \mathbb{R}^{I_1\times I_2 \times \cdots \times I_N\times m}$ is an $(N+1)$-mode tensor with the column tensors $\mathscr{B}_1,\mathscr{B}_2,\ldots,\mathscr{B}_m\in \mathbb{R}^{I_1\times I_2 \times \cdots \times I_N}$ and
		$z=(z_1,z_2,\ldots,z_m)^T\in \mathbb{R}^m$. For an arbitrary $(N+1)$-mode tensor $\mathscr{A}$ with $N$-mode column tensors $\mathscr{A}_1,\mathscr{A}_2,\ldots,\mathscr{A}_m$, the following statement holds
		\begin{equation}
		\mathscr{A} \boxtimes^{(N+1)} (\mathscr{B} \bar{\times}_{_{N+1}} z) = (\mathscr{A} \boxtimes^{(N+1)} \mathscr{B}) z.
		\end{equation}
	\end{proposition}

	\section{On the sensitivity analysis of Stein tensor equation}\label{Sec2}
	
	In this section, we mainly discuss on conditioning of Stein tensor equation
	\eqref{eql1}. To this end, first, we consider a linear system of equations
	which is equivalent to \eqref{eql1} and then derive some lower and upper bounds for the condition number of the coefficient matrix of the  linear system of equations.
	
\noindent 	It is well-known that  \eqref{eq1} is equivalent to the linear system of equations,
	\begin{equation}\label{vecsys}
	\tilde{\mathcal{A}} x=b,
	\end{equation}
	with $x=\textrm{vec} (\mathscr{X})$, $b=\textrm{vec} (\mathscr{D})$, and
	\begin{equation}\label{kron}
	\tilde{\mathcal{A}}=\sum\limits_{j=1}^{N} {I^{(I_N)}\otimes \cdots \otimes I^{(I_{j+1})}\otimes A^{(j)}\otimes I^{(I_{j-1})}\otimes \cdots \otimes I^{(I_1)}},
	\end{equation}
	In addition, it can be observed that
	\begin{equation*}
	\mathscr{Y} = \mathscr{X}{ \times _1}{A^{(1)}}{ \times _2}{A^{(2)}} \cdots { \times _N}{A^{(N)}} \quad \Leftrightarrow \quad {Y_{(1)}} = {A^{(1)}}{X_{(n)}}{({A^{(N)}} \otimes  \cdots  \otimes {A^{(2)}})^T}.
	\end{equation*}
	In view of the above relation, we deduce that \eqref{eql1} corresponds to the following linear system of equations,
	\[\mathcal{A}x:=\left( {I - {A^{(N)}} \otimes  \cdots  \otimes {A^{(2)}} \otimes {A^{(1)}}} \right)\textrm{vec}(\mathscr{X}) =\textrm{vec}(\mathscr{F}). \]
	As a result,  in view of the fact that``${\left\| {\mathscr{X}} \right\|}={\left\| {\textrm{vec}(\mathscr{X})} \right\|_2}$", the sensitivity analyses of \eqref{eq1} and \eqref{eql1} are closely related to deriving bounds for condition numbers of $\tilde{\mathcal{A}}$ and $\mathcal{A}$, respectively. Basically,
	for linear system of equations $\mathcal{A}x=b$ and $\mathcal{A}({x+\Delta x})={b+\Delta b}$, we know that
	\begin{equation}\label{Dec23-2}
	\frac{{\left\| {\Delta x} \right\|_2}}{{\left\| x \right\|_2}} \le \mathrm{cond}(\mathcal{A})\frac{{\left\| {\Delta b} \right\|_2}}{{\left\| b \right\|_2}}.
	\end{equation}
	Also, under the assumption $\left\| {{\mathcal{A}^{ - 1}}} \right\|_2\left\| {\Delta \mathcal{A}} \right\|_2 < 1$, for the linear system of equations
	\[({\mathcal{A}+\Delta \mathcal{A}})({x+\Delta x})={b+\Delta b},\]
	the following result exists in the literature
	\[\frac{{\left\| {\Delta x} \right\|_2}}{{\left\| x \right\|_2}} \le \frac{{\mathrm{cond}(\mathcal{A})}}{{1 - \mathrm{cond}(\mathcal{A})\frac{{\left\| {\Delta \mathcal{A}} \right\|_2}}{{\left\| \mathcal{A} \right\|_2}}}} \left\{
	\frac{{\left\| {\Delta \mathcal{A}} \right\|_2}}{{\left\| \mathcal{A} \right\|_2}} +\frac{{\left\| {\Delta b} \right\|_2}}{{{\left\| b \right\|}_2}}  \right\},\]
	one may refer to \cite{Golubbook} for further details about perturbation analysis for linear system of equations.
	
	\noindent In \cite{Najafi}, some lower and upper bounds for  $\tilde{\mathcal{A}}$ has been derived under certain conditions. Therefore, in the sequel, we assume that $\mathcal{A}$ is invertible and limit the discussions to deriving bounds for  $cond(\mathcal{A})$.
	
	\noindent In \cite{Xu}, it is shown that
	\[cond(\mathcal{A}) \ge \frac{{{{\max }_{{\lambda _{{i_k}}} \in \sigma ({A^{(k)}})}}\left| {1 - {\lambda _{{i_1}}}{\lambda _{{i_2}}} \ldots {\lambda _{{i_N}}}} \right|}}{{{{\max }_{{\lambda _{{i_k}}} \in \sigma ({A^{(k)}})}}\left| {1 - {\lambda _{{i_1}}}{\lambda _{{i_2}}} \ldots {\lambda _{{i_N}}}} \right|}}.\]
	Furthermore, for the case $\|\mathcal{A}\|_2 < 1$, the following upper bound for the condition number is also presented
	\[cond(\mathcal{A}) \le \frac{{1 + \prod\nolimits_{i = 1}^N {{{\| {{A^{(i)}}} \|}_2}} }}{{1 - \prod\nolimits_{i = 1}^N {{{\| {{A^{(i)}}}\|}_2}} }}.\]
	
	\noindent Now, we start our results by establishing the following proposition which presents
	an upper bound for the condition number of $\mathcal{A}$ under certain condition.
	
	\begin{proposition}\label{prop1}
		Assume that  $\prod\nolimits_{i = 1}^N {\sigma _{\min }} ({A^{(i)}})>1$, then
		\[cond(\mathcal{A}) \le \left(\frac{\prod\nolimits_{i = 1}^N {\sigma _{\min }} ({A^{(i)}}){ }}{{\prod\nolimits_{i = 1}^N {\sigma _{\min }} ({A^{(i)}}){ } - 1 }}\right) 	\left(1+\prod\nolimits_{i = 1}^N {{{\| {{A^{(i)}}} \|}_2}}\right) .\]
	\end{proposition}	
	
	\begin{proof}
		For simplicity, let $\mathcal{F}={A^{(N)}} \otimes  \cdots  \otimes {A^{(1)}}$.	
		It is immediate to conclude that
		\begin{eqnarray}
		\nonumber \|\mathcal{A}\|_2 &\le& 1 + \| \mathcal{F} \|_2=  1 + \sqrt{\rho\left( \mathcal{F} \mathcal{F}^T\right)}\\
		\nonumber & = & 1+ {\prod\limits_{i = 1}^N {\sigma _{\max }} ({A^{(i)}})},\\
		& = & 1+ \prod\nolimits_{i = 1}^N {{{\| {{A^{(i)}}} \|}_2}}. \label{144}
		\end{eqnarray}	
		Evidently, we have $(I-\mathcal{F})^{-1} = -(I-\mathcal{F}^{-1})^{-1}\mathcal{F}^{-1}.$  It is well-known that
		\[
		\mathcal{F}^{-1}=({A^{(N)}})^{-1} \otimes  \cdots  \otimes ({A^{(1)}})^{-1}.
		\] From the above relation and the fact that
		\[
		\|\mathcal{F}^{-1}\|_2=\prod\nolimits_{i = 1}^N \|(A^{(i)})^{-1}\|_2= {(\prod\nolimits_{i = 1}^N {\sigma _{\min }} ({A^{(i)}}){ }})^{-1} < 1,
		\]
		we get,
		\[
		\| (I-\mathcal{F})^{-1}\|_2 \le \| (I-\mathcal{F}^{-1})^{-1}\|_2 \| \mathcal{F}^{-1}\|_2 \le \| (I-\mathcal{F}^{-1})^{-1}\|_2 \le \frac{1}{1-\|\mathcal{F}^{-1}\|_2 },
		\]
		Now we can conclude the result immediately.
	\end{proof}
	
	\noindent For deriving alternative bounds for $cond(\mathcal{A})$, we first prove the following two propositions.

	\begin{proposition}\label{p4a}
		Let $A^{(i)}\in\mathbb{R}^{n_i\times n_i}$ and  $x_i\in \mathbb{R}^{n_i}$ for $i=1,2,\ldots,\ell$, then
		\begin{equation}\label{eq62}
		(\mathop  \bigotimes \limits_{i = 1}^\ell x_i)^T\mathcal{H}(A^{(1)}\otimes A^{(2)}\otimes \ldots \otimes A^{(\ell)}) \mathop  \bigotimes \limits_{i = 1}^\ell x_i = \prod\limits_{i=1}^{\ell} {x_i^T\mathcal{H}(A^{(i)}) x_i}.
		\end{equation}
	\end{proposition}
	
	\begin{proof}
		We prove the assertion by induction. For $\ell=2$, using the fact that $x_i^T\mathcal{S}(A^{(i)})x_i=0$ (for $i=1,2$),  we can conclude the result from the following equality (see \cite{Zak})
		\[
		\mathcal{H}(A^{(1)}\otimes A^{(2)}) = \mathcal{H}(A^{(1)})\otimes \mathcal{H}(A^{(2)})-
		\mathcal{S}(A^{(1)})\otimes\mathcal{S}(A^{(2)}).
		\]
		Assume that \eqref{eq62} is true for $\ell=k$. Now for $\ell=k+1$, setting
		\[
		\mathcal{Y}_k=\mathop  \bigotimes \limits_{i = 2}^{(k+1)} x_i \quad \mathcal{Y}_{k+1}=x_1\otimes\mathcal{Y}_{k}, \quad \text{and} \quad  \mathcal{A}_k=A^{(2)}\otimes \ldots \otimes A^{(k+1)},
		\]
		we get
		\begin{eqnarray*}
			\mathcal{Y}_{k+1}^T \mathcal{H}(A^{(1)}\otimes A^{(2)}\otimes \ldots \otimes A^{(\ell)})\mathcal{Y}_{k+1} & =&(x_1 \otimes \mathcal{Y}_k)^T\mathcal{H}(A^{(1)}\otimes \mathcal{A}_k)(x_1 \otimes \mathcal{Y}_k)\\
			& = & (x_1^T \mathcal{H}(A^{(1)}) x_1) \times (\mathcal{Y}_k^T\mathcal{H}(\mathcal{A}_k)\mathcal{Y}_k).
		\end{eqnarray*}
		Using the assumption of induction for the term $\mathcal{Y}_k^T\mathcal{H}(\mathcal{A}_k)\mathcal{Y}_k$, we can conclude the result  immediately.
	\end{proof}

	\begin{proposition}\label{prop2}
		Assume that $\mathcal{A}=I - {A^{(N)}} \otimes  \cdots  \otimes {A^{(2)}} \otimes {A^{(1)}}$. Then,
		\begin{equation}
		{\lambda _{\max }}(\mathcal{A}\mathcal{A}^T) \ge 1 + 	\prod\nolimits_{i = 1}^N {\sigma _{\max }^2} ({A^{(i)}})-2\prod\nolimits_{i = 1}^N {y_i^T\mathcal{H}} ({A^{(i)}}){y_i},
		\end{equation}
		and
		\begin{equation}
		{\lambda _{\min }}(\mathcal{A}\mathcal{A}^T) \le 1 + 	\prod\nolimits_{i = 1}^N {\sigma _{\min }^2} ({A^{(i)}}) -2 	\prod\nolimits_{i = 1}^N {z_i^T\mathcal{H}} ({A^{(i)}}){z_i},
		\end{equation}
		where $A^{(i)}(A^{(i)})^Tz_i = \sigma^2_{\min}(A^{(i)}) z_i$ and	 $A^{(i)}(A^{(i)})^Ty_i = \sigma^2_{\max}(A^{(i)}) y_i$ with ${\| z_i\|}_2=1$ and ${\| y_i\|}_2=1$ for
		$i=1,2,\ldots,N$.
	\end{proposition}

	\begin{proof}
		It is not difficult to verify that
		\begin{eqnarray}
		\nonumber \mathcal{A}{\mathcal{A}^T} & = & (I - {A_N} \otimes  \cdots  \otimes {A_1})(I - A_N^T \otimes  \cdots  \otimes A_1^T)\\
		& = &  I + {A_N}A_N^T \otimes  \cdots  \otimes {A_1}A_N^T - 2\mathcal{H}({A_N} \otimes  \cdots  \otimes {A_1}). \label{141}
		\end{eqnarray}	
		Setting $\mathcal{Y}=(y_N\otimes \cdots \otimes  y_1)$  and $\mathcal{Z}=(z_N\otimes \cdots \otimes  z_1)$, in view of Proposition \ref{p4a}, we obtain
		\[
		\mathcal{Y}^T\mathcal{A}{\mathcal{A}^T}\mathcal{Y} = 1 + 	\prod\nolimits_{i = 1}^N {\sigma _{\max }^2} ({A^{(i)}})-2\prod\nolimits_{i = 1}^N {y_i^T\mathcal{H}} ({A^{(i)}}){y_i},
		\]
		and
		\[
		\mathcal{Z}^T\mathcal{A}{\mathcal{A}^T}\mathcal{Z} = 1 + 	\prod\nolimits_{i = 1}^N {\sigma _{\min }^2} ({A^{(i)}}) -2 	\prod\nolimits_{i = 1}^N {z_i^T\mathcal{H}} ({A^{(i)}}){z_i},
		\]
		which completes the proof immediately.
	\end{proof}
	
	\begin{remark} If the matrices $A^{(i)}$s for $i=1,2,\ldots,N$ are all positive definite, then
		\[
		{\lambda _{\min }}(\mathcal{A}\mathcal{A}^T) \le 1+ 	\prod\nolimits_{i = 1}^N {\sigma _{\min }^2} ({A^{(i)}}).
		\]
		Furthermore, if we have 	
		\[
		\prod\nolimits_{i = 1}^N {\sigma _{\max }^2} ({A^{(i)}})\ge 2\prod\nolimits_{i = 1}^N \lambda _{\max }({\mathcal{H}} ({A^{(i)}})),
		\]
		then the following upper bound can be derived immediately from Proposition \ref{prop2},
		\[cond(\mathcal{A}) \ge \frac{\sqrt{1+\prod\nolimits_{i = 1}^N {\sigma _{\max }^2} ({A^{(i)}})- 2\prod\nolimits_{i = 1}^N \lambda _{\max }({\mathcal{H}} ({A^{(i)}}))}}{\sqrt{1+ 	\prod\nolimits_{i = 1}^N {\sigma _{\min }^2} ({A^{(i)}}) }}\ge \frac{1}{\sqrt{1+ 	\prod\nolimits_{i = 1}^N {\sigma _{\min }^2} ({A^{(i)}}) }}.\]
	\end{remark}
	
\noindent Here we recall a useful proposition which is a consequence of Weyl's Theorem, see \cite[Theorem 4.3.1]{Horn}.
	
	\begin{proposition} \label{prop2.1}
		Suppose that $ A, B \in\mathbb{R}^{n\times n} $ are two symmetric matrices. Then,
		\begin{align*}
		\lambda_{\max}(A + B) &\leq \lambda_{\max}(A) +  \lambda_{\max}(B),\\
		\lambda_{\min}(A + B) &\geq \lambda_{\min}(A) + \lambda_{\min}(B).
		\end{align*}
	\end{proposition}

\noindent 	Using Proposition \ref{prop2.1} and some straightforward algebraic computations, we can prove the following result.
	
	\begin{proposition}
		Let $\mathcal{F}={A^{(N)}} \otimes  \cdots  \otimes {A^{(1)}}.$  Assume that $r$ is an even number and $\lambda \in \sigma (\mathcal{H}(\mathcal{F}))$, then
		\[\left| {\lambda (\mathcal{H}(\mathcal{F}))} \right| \le \sum\limits_{r = 0}^N {\frac{{N!}}{{r!\left( {N - r} \right)!}}M_s^rM_H^{N - r}} \le (M_r+M_H)^N. \]
		where
		\[{M_S} = \mathop {\max }\limits_{i = 1,2,...,N} {\| {\mathcal{S}({A^{(i)}})} \|_2}\quad \text{and} \quad {M_H} = \mathop {\max }\limits_{i = 1,2,...,N} {{\rho(\mathcal{H}({A^{(i)}}))}}. \]
		Here, for a given matrix $W$, the notation $\rho(W)$ stands for the spectral radius of $W$.
	\end{proposition}
	
	\begin{remark}
		A simple conclusion of the above proposition is that if $M_r+M_H < 1 $ then
		the matrix $\mathcal{A}$ is positive definite, i.e., $\mathcal{H}(\mathcal{A})$ is a symmetric positive definite. In this case, we can obtain an upper bound for $\|\mathcal{A}^{-1}\|_2$. In fact, from \eqref{141}, it can be seen that
		\begin{eqnarray*}
			{\lambda _{\min }}(\mathcal{A}{\mathcal{A}^T}) &\ge& 1 + \prod\limits_{i = 1}^N {\sigma _{\min }^2} ({A^{(i)}}) - 2{\lambda _{\min }}(\mathcal{H}({A^{(N)}} \otimes  \cdots  \otimes {A^{(1)}}))\\
			&\ge& 1 + \prod\limits_{i = 1}^N {\sigma _{\min }^2} ({A^{(i)}}) - {({M_S} + {M_H})^N}\\
			&\ge& \prod\limits_{i = 1}^N {\sigma _{\min }^2} ({A^{(i)}}).
		\end{eqnarray*}
		Therefore, we have
		\begin{equation}\label{142}
		\|\mathcal{A}^{-1}\|_2 \le \frac{1}{{\prod\limits_{i = 1}^N {\sigma _{\min }}. ({A^{(i)}})}}
		\end{equation}
		Now, in view of inequality \eqref{144} together with \eqref{142} gives an upper bound for the condition number of $\mathcal{A}$ as follows:
		\[
		cond(\mathcal{A})  \le \frac{1+ {\prod\limits_{i = 1}^N {\sigma _{\max }} ({A^{(i)}})}}{{\prod\limits_{i = 1}^N {\sigma _{\min }} ({A^{(i)}})}}.
		\]
	\end{remark}
	
\noindent	We end this part by the following remark which is an observation for the case
	that $A^{(i)}$'s for $i=1,2,\ldots,N$ are all diagonalizable.

	\begin{remark}
		Let $A^{(i)}$ be a diagonalizable matrix, i.e, there exists nonsingular matrix $S_i$
		associated with $A^{(i)}$ such that $A^{(i)}=S_iD_iS_i^{-1}$ for $i=1,2,\ldots,N$. Setting $\mathcal{S}={S_N} \otimes  \cdots  \otimes {S_1}$, we have
		$\mathcal{A}=\mathcal{S}(I-D_N \otimes  \cdots  \otimes D_1)\mathcal{S}^{-1}$. Hence, if $1\notin \sigma ({A^{(N)}} \otimes  \cdots  \otimes {A^{(1)}} )$ then
		\[
		\mathcal{A}^{-1}=  \mathcal{S}^{-1} (I-D_N \otimes  \cdots  \otimes D_1)^{-1}\mathcal{S}.
		\]
		As a result, we get
		\[
		\|\mathcal{A}^{-1} \|_2 \le \prod\limits_{i = 1}^N \|S_i^{-1}\|_2\|S_i\|_2 M_D=\prod\limits_{i = 1}^N cond(S_i) M_D,
		\]
		where
		\[M_D=\max \left\{ {\frac{1}{{\left| {1 - {\lambda _{\min }}({D_N} \otimes  \cdots  \otimes {D_1})} \right|}},\frac{1}{{\left| {1 - {\lambda _{\max }}({D_N} \otimes  \cdots  \otimes {D_1})} \right|}}} \right\}.\]
		In this case, we have the following inequality
		\[
		cond(\mathcal{A})  \le \prod\limits_{i = 1}^N cond(S_i) (1+ \prod\limits_{i = 1}^N {\sigma _{\max }}({A^{(i)}})) M_D.
		\]
		Notice that analogous to the proof of Proposition \ref{prop1}, in the case that
		$\prod\limits_{i = 1}^N \|D_i^{-1}\|_2 <1$, we have
		\[
		cond(\mathcal{A})  \le \prod\limits_{i = 1}^N cond(S_i) (1+ \prod\limits_{i = 1}^N {\sigma _{\max }}({A^{(i)}}))\times \frac{1}{1-\prod\limits_{i = 1}^N \|D_i^{-1}\|_2}.
		\]
		In addition, with the similar strategy used in \cite{Xu}, if $\prod\limits_{i = 1}^N \|D_i\|_2 <1$ then
		\[
		cond(\mathcal{A})  \le \prod\limits_{i = 1}^N cond(S_i) (1+ \prod\limits_{i = 1}^N {\sigma _{\max }}({A^{(i)}}))\times \frac{1}{1-\prod\limits_{i = 1}^N \|D_i\|_2}.
		\]
		
		Finally, comment that if the matrices $D_i$ are all positive definite matrices ($i=1,2,\ldots,N$) then
		\[
		{\lambda _{\min }}({D_N} \otimes  \cdots  \otimes {D_1}) = \prod\limits_{i = 1}^N {\lambda _{\min }(D_i)} \quad \text{and} \quad {\lambda _{\max }}({D_N} \otimes  \cdots  \otimes {D_1}) = \prod\limits_{i = 1}^N {\lambda _{\max }(D_i)}.
		\]	
	\end{remark}

	\section{Tensor form of GKB and Gauss-type quadrature}\label{Sec3}
	In this section, we briefly describe the implantation of GKB process in tensor framework. For simplicity, in the sequel, we use two linear
	operators $ \tilde{\mathcal{M}}, {\mathcal{M}}:\mathbb{R}^{I_1\times I_2 \times \cdots \times I_N}   \to  \mathbb{R}^{I_1\times I_2 \times \cdots \times I_N} $ such that
	\begin{eqnarray}
	\tilde{\mathcal{M}}(\mathscr{X}) & := &\mathscr{X}\times_1 A^{(1)}+\mathscr{X}\times_2 A^{(2)}+\cdots+\mathscr{X}\times_N A^{(N)}, \label{linop1}\\
	\mathcal{M}(\mathscr{X}) &:=&\mathscr{X}-\mathscr{X}\times_1 A^{(1)}\times_2 A^{(2)}\ldots\times_N
	A^{(N)}. \label{linop2}
	\end{eqnarray}
	The adjoint of $\tilde{\mathcal{M}}$ and $\mathcal{M}$ are respectively given by
	\begin{eqnarray*}
		\tilde{\mathcal{M}}^*(\mathscr{Y}) & := &\mathscr{Y}\times_1 (A^{(1)})^T+\mathscr{Y}\times_2 (A^{(2)})^T+\cdots+\mathscr{Y}\times_N (A^{(N)})^T, \\
		\mathcal{M}^*(\mathscr{Y}) & := &\mathscr{Y}-\mathscr{Y}\times_1 (A^{(1)})^T\times_2 (A^{(2)})^T\ldots\times_N
		(A^{(N)})^T,
	\end{eqnarray*}
	for $\mathscr{Y} \in \mathbb{R}^{I_1\times I_2 \times \cdots \times I_N} $. Using the linear operators \eqref{linop1} and \eqref{linop2}, the tensor equations
	\eqref{eq1} and \eqref{eql1} are respectively written by
	\[\tilde{\mathcal{M}}(\mathscr{X})=\mathscr{D} \quad \text{and} \quad \mathcal{M}(\mathscr{X})=\mathscr{F}.\]
	We comment that all of the results in this section can be applied for any other linear operator from $\mathbb{R}^{I_1\times I_2 \times \cdots \times I_N}$  to $\mathbb{R}^{I_1\times I_2 \times \cdots \times I_N}.$
	
\noindent Consider the linear system of equation $Ax=b$ where $A\in \mathbb{R}^{n\times n}$. We recall that the well-known GKB process, applied to the matrix $A$, produces the decomposition  $V^TAU=T$ where $V$ and $U$ are orthogonal matrices and $T$ is a bidiagonal matrix.	
%
\noindent It is natural to use the process for an arbitrary linear operator over $\mathbb{R}^{I_1\times I_2 \times \cdots \times I_N}.$  The corresponding approach is called GKB based on tensor format (GKB$_{-}$BTF) which is summarized in Algorithm \ref{a2}.
	
	\RestyleAlgo{ruled}
	\LinesNumbered
	\begin{algorithm}
		\caption{GKB$_{-}$BTF process associated with linear operators ${\mathcal{M}}$ (and $\tilde{\mathcal{M}}$). \label{a2}}
		\textbf{Input:} Linear operator $\mathcal{M}$ ($\tilde{\mathcal{M}}$) and the right-hand side $\mathscr{F}$  ($\mathscr{D}$).\\
		Set $\beta_1=\|\mathscr{F}\|$, $\mathscr{V}_1=\frac{1}{\beta}\mathscr{F}$,  $\mathscr{U}={\mathcal{M}}^*(\mathscr{V}_1)$ , $\alpha_1=\|\mathscr{U}\|$, and $\mathscr{U}_1=\frac{1}{\alpha_1}\mathscr{U}$.\label{l1}\\
		\Begin{
			\For {$j=2,\ldots,m$}{
				
				$\mathscr{V}={\mathcal{M}}(\mathscr{U}_{j-1})-\alpha_{j-1}\mathscr{V}_{j-1}$\label{l5}\;
				$\beta_j=\|\mathscr{V}\|$\;\label{l6}
				\If{$\beta_j=0$}{Stop}
				$\mathscr{V}_j=\mathscr{V}/\beta_j$ \label{l10}\;
				$\mathscr{U}={\mathcal{M}}^*(\mathscr{V}_j)-\beta_j\mathscr{U}_{j-1}$\label{l11}\;
				$\alpha_j=\|\mathscr{U}\|$\;\label{l12}
				\If{$\alpha_j=0$}{Stop}
				$\mathscr{U}_{j}=\mathscr{U}/\alpha_j$\label{l16}\;
			}
		}
	\end{algorithm}

\noindent In  Algorithm \ref{a2}, suppose that $m=k+1$. Moreover, assume that there is no break-down in the algorithm and let $\bar{T}_k$ be an $(k+1)\times k$ lower bidiagonal matrix whose nonzero entries are those computed in Lines \ref{l6} and \ref{l12} of Algorithm \ref{a2}. In the following, the matrix $T_k$ stands for the $k\times k$ matrix extracted from $\bar{T}_k$ as follows:
	\[\bar{T}_k = \left( {\begin{array}{*{20}{c}}
		{{T_k}}\\
		{{\beta _{k + 1}}e_{k}^T}
		\end{array}} \right).\]

	\begin{theorem}\label{thm1}
		Let ${\tilde{\mathscr{V}}}_k$, ${\tilde{\mathscr{U}}}_k$, ${\tilde{\mathscr{W}}}_k$ and ${\tilde{\mathscr{W}}}^*_k$  be the $(N+1)$-mode tensors with frontal slices ${{\mathscr{V}}}_j$, ${{\mathscr{U}}}_j$, $\mathscr{W}_j:=\mathcal{M}(\mathscr{U}_j)$ and $\mathscr{W}_j^*:=\mathcal{M}^*(\mathscr{V}_j)$  for $j=1,\ldots,k$ computed by Algorithm \ref{a2}.
		Then the following statements hold
		\begin{eqnarray}
		\nonumber {\tilde{\mathscr{W}}}_k&=&{\tilde{\mathscr{V}}}_{k}\times_{(N+1)}T_k^T+\beta_{k+1}\mathscr{Z}\times_{(N+1)}E_k,\\
		& =& {\tilde{\mathscr{V}}}_{k+1}\times_{(N+1)}\bar{T}_k^T \label{moh}\\
		{\tilde{\mathscr{W}}}_k^*&=&\mathcal{\tilde{\mathscr{U}}}_k\times_{(N+1)}T_k,\label{moh1}
		\end{eqnarray}
		in which $\mathscr{Z}$ is an $(N+1)-$mode tensor with ``$k$" column tensors $0,\ldots,0,\mathscr{V}_{k+1}$ and $E_k$ is an $k\times k$ matrix of the form $E_k=[0,\ldots,0,e_k]$ where $e_k$ is the $k$th column of the identity matrix of order $k$.
	\end{theorem}
	
	\begin{proof}
		From Lines \ref{l5} and \ref{l10},
		\begin{equation}\label{171}
		\mathcal{M}(\mathscr{U}_{j-1})=\alpha_{j-1}\mathscr{V}_{j-1}+\beta_j\mathscr{V}_j.
		\end{equation}
		Note that the $(j-1)$th frontal slice of \eqref{moh} is given by
		\begin{eqnarray}
		\nonumber ({\tilde{\mathscr{V}}}_{k+1}\times_{(N+1)}T_k^T )_{:\ldots:(j-1)}&=&\sum\limits_{\ell=1}^{k} { \mathscr{V}_\ell (T_{k})_{\ell,j-1}}\\
		& = & \alpha_{j-1} \mathscr{V}_{j-1} + \beta_j \mathscr{V}_j. \label{172}
		\end{eqnarray}
		In view of \eqref{171} and \eqref{172}, we can conclude the validity of \eqref{moh}.
		To derive \eqref{moh1}, one may first notice that Lines \ref{l1}, \ref{l11} and \ref{l16} gives
		\begin{equation*}
		\mathcal{M}^*(\mathscr{V}_{j})=\beta_j\mathscr{U}_{j-1}+\alpha_{j}\mathscr{U}_{j},\quad j=1,2,\ldots,
		\end{equation*}
		where $\mathscr{U}_{-1}$ is assumed to be zero.
		Now considering the $j$ frontal slice of the right-hand side of \eqref{moh1}, we can deduce the second assertion.	
	\end{proof}
	
	\begin{remark}
		It is obvious that one may state the above theorem for any linear operator over $\mathbb{R}^{I_1\times I_2 \times \cdots \times I_N}$ instead of ${\mathcal{M}}(\cdot)$. In what follows, the results are stated for ${\mathcal{M}}(\cdot)$ and it should be commented that all of  results remain true, if we replace ${\mathcal{M}}(\cdot)$ by $\tilde{\mathcal{M}}(\cdot)$ or any other linear operators over $\mathbb{R}^{I_1\times I_2 \times \cdots \times I_N}$.
	\end{remark}
	
\noindent	Let the linear system associated with \eqref{eql1} be extremely ill-conditioned.
	In the case that the right-hand side of \eqref{eql1} contains some noise,  it is inefficient to approximate the solution of \eqref{eql1} without any regularization technique. To overcome this, we may use Tikhonov regularization which consists of
	solving the following minimization problem,
	\begin{equation}\label{eq33}
	\min\left\{  \left\| \mathcal{M}(\mathscr{X})-\mathscr{F}
	\right\|^2 +\mu\|\mathscr{X}\|^2 \right\},
	\end{equation}
	(over ${\mathscr{X}\in \mathbb{R}^{I_1\times I_2 \times \ldots \times I_N}}$) instead of solving $ \mathcal{M}(\mathscr{X})=\mathscr{F}$ in which $ \mu> 0 $ is called the regularization parameter.
	
	Let $\mathscr{X}_{k,\mu_k}= {\tilde{\mathscr{V}}}_{k}\bar{\times}_{(N+1)} y_{k,\mu_k}$
	be an approximate solution where $ {\tilde{\mathscr{V}}}_{k}$ is defined as before.
	From \eqref{moh}, by Lemma \ref{lem3.1} and Proposition \ref{p2}, we have
	
	\begin{eqnarray}
	\nonumber \left\| \mathcal{M}(\mathscr{X})-\mathscr{F}
	\right\| & = & \left\|  {\tilde{\mathscr{V}}}_{k+1}\times_{(N+1)}T_k^T\bar{\times}_{(N+1)} y_{k,\mu_k} -\mathscr{F}
	\right\| \\
	\nonumber & = & \left\|  {\tilde{\mathscr{V}}}_{k+1}\times_{(N+1)}T_k y_{k,\mu_k} -\mathscr{F}
	\right\|\\
	\nonumber  & = & \left\|   \tilde{\mathscr{V}}_{k+1}\boxtimes^{(N+1)} ({\tilde{\mathscr{V}}}_{k+1}\times_{(N+1)}\bar{T}_k y_{k,\mu_k} -\mathscr{F})
	\right\|_2\\
	\nonumber   & = & \left\|   (\tilde{\mathscr{V}}_{k+1}\boxtimes^{(N+1)} {\tilde{\mathscr{V}}}_{k+1})\bar{T}_k y_{k,\mu_k} -\tilde{\mathscr{V}}_{k+1}\boxtimes^{(N+1)}\mathscr{F}
	\right\|_2\\
	& = & \left\| \bar{T}_k y_{k,\mu_k} -\beta_1 e_1
	\right\|_2,\label{eq201}
	\end{eqnarray}
	which shows that \eqref{eq33} is equivalent to the following low dimensional minimization problem,

	\begin{eqnarray}
	\mathop {\min }\limits_{y \in {\mathbb{R}^k}}\left\{  \left\| \bar{T}_k y -\beta_1 e_1 \right\|_2^2 +\mu\|y\|^2_2 \right\}&=&\mathop {\min }\limits_{y \in {\mathbb{R}^k}} {\left\| {\left( {\begin{array}{*{20}{c}}
				{{{\bar{ T }}_k}}\\
				{\sqrt \mu  I}
				\end{array}} \right)y - \left( {\begin{array}{*{20}{c}}
				{{\beta _1}}\\
				0
				\end{array}} \right)} \right\|_2}.\label{eq33e}
	\end{eqnarray}
	As a result, the solution of \eqref{eq33e} is given by
	\[
	y_{k,\mu}=\beta_1 (\bar{T}_k^T\bar{T}_k+\mu I)^{-1} \bar{T}_k^Te_1.
	\]
	Consequently, we have
	\begin{eqnarray*}
		\left\| \bar{T}_k y_{k,\mu} -\beta_1 e_1 \right\|_2 & = & \left\| \beta_1\bar{T}_k  (\bar{T}_k^T\bar{T}_k+\mu I)^{-1} \bar{T}_k^Te_1 -\beta_1 e_1 \right\|_2\\
		& = & \left\| (\bar{T}_k  (\bar{T}_k^T\bar{T}_k+\mu I)^{-1} \bar{T}_k^T -I)\beta_1 e_1 \right\|_2\\
		& = & \left\|  (\mu^{-1}\bar{T}_k \bar{T}_k ^T+I)^{-1} \beta_1 e_1\right\|_2\\
		& = &\beta_1 \sqrt{e_1^T (\mu^{-1}\bar{T}_k \bar{T}_k ^T+I)^{-2} e_1}.
	\end{eqnarray*}
	Therefore, if we define the function $\phi_k (\mu)$ by
	\begin{equation}\label{eq20}
	\phi_k (\mu) = \beta_1^2 e_1^T (\mu^{-1}\bar{T}_k \bar{T}_k ^T+I)^{-2} e_1,
	\end{equation}
	we can conclude the following result.
	
	\begin{proposition}\label{prop7}
		Assume that $\eta$ and $\epsilon$ are positive constants such that $\eta>1$.	Let  $\phi_k (\mu)$ be defined by \eqref{eq20}. Then any solution $\mu>0$ of
		$\phi_k (\mu)$ satisfying
		\[
		\epsilon^2 \le \phi_k (\mu)\le \eta^2 \epsilon^2,
		\]
		determines a solution $y_{k,\mu_k}$ of \eqref{eq33e} such that
		\[
		\epsilon \le \left\| \bar{T}_k y_{k,\mu} -\beta_1 e_1 \right\|_2\le \eta \epsilon,
		\]	
		and $\mathscr{X}_{k,\mu_k}= {\tilde{\mathscr{V}}}_{k}\bar{\times}_{(N+1)} y_{k,\mu_k}$ satisfies
		\begin{equation}
		\epsilon \le \left\| \mathcal{M}(\mathscr{X}_{k,\mu_k})-\mathscr{F}
		\right\|\le \eta \epsilon. \label{eq202}
		\end{equation}
	\end{proposition}
	\begin{proof}
		Eq. \eqref{eq202} follows from \eqref{eq201} immediately.
	\end{proof}

	\begin{proposition}\label{prop8}
	Let  $\phi_k (\mu)$ be defined by \eqref{eq20}. Then the function
		$\mu \to \phi_k (1/\mu)$ is strictly decreasing and convex for $\mu > 0$. Moreover,
		\[\mathop {\lim }\limits_{\mu  \to \infty } {\phi _k}(1/\mu ) = \beta _1^2.\]
		In particular, Newton’s method applied to the solution of the equation
		${\phi _k}(1/\mu )=\eta^2 \epsilon^2$ with initial approximate solution $\mu_0$ to  the left of solution converges monotonically and quadratically.
	\end{proposition}
	\begin{proof}
		See \cite[Proposition 3.6]{Buccini}.
		\end{proof}
	
\noindent	For simplicity, we set $\nu =\mu^{-1}.$ Consider the integral 	
	\[I(f):=\int {f(t)d\omega (t)}, \]
	for suitable function $f$ and assume that $\mathcal{G}_k$ and $\mathcal{R}_{k+1}$ are respectively the $k$-point Gauss quadrature  and $(k+1)$-point Gauss-Radau rule. In \cite{Bentbib}, it has been discussed that using spectral factorization of $\bar{T}_k\bar{T}_k^T$, the function $\phi_k(\nu)= \beta_1^2 e_1^T (\nu \bar{T}_k \bar{T}_k ^T+I)^{-2} e_1$ can be expressed by
	\[\phi_k(\nu):=\int {f_\nu(t)d\omega (t)},\]
	with $f_\nu (t):=(\nu t +1)^{-2}$.
	
\noindent	Analogous to \cite{Bentbib}, we can deduce that
	\[
	\mathcal{G}_kf_\nu =\beta_1^2 e_1^T(\nu {T}_k{T}_k^T+I_k)^{-2}e_1^T \quad \text {and} \quad
	\mathcal{R}_{k+1}f_\nu =\beta_1^2 e_1^T(\nu \bar{T}_k\bar{T}_k^T+I_{k+1})^{-2}e_1^T.
	\]
	It is known from \cite{Bentbib2016} that
	\[
	\mathcal{G}_1 f_\nu < \cdots < \mathcal{G}_{k-1} f_\nu  < \mathcal{G}_k f_\nu  < \phi_k(\nu)
	\]
	and
	\[
	\phi_k(\nu) < \mathcal{R}_{k+1} f_\nu < \mathcal{R}_{k} f_\nu < \cdots < \mathcal{R}_1 f_\nu.
	\]
	In fact the above two relations show that the $\mathcal{G}_kf_\nu $ and $\mathcal{R}_{k+1} f_\nu$ provide lower and upper bounds for $\phi_k(\nu)$ (or $\phi_k(\mu)$ with $\mu=1/\nu$). The bounds are helpful for determining $\mu$ by the discrepancy principle in an inexpensive way. To this end, at step $k\ge 2$,  we find $\nu > 0$ by solving the following nonlinear equation,
	\begin{equation}\label{eq203}
	\mathcal{G}_k f_\nu = \epsilon^2.
	\end{equation}
	We comment that in view of Proposition \ref{prop8}, one may use Newton's method efficiently to solve \eqref{eq203}. If for the solution $\nu$, we have
	\begin{equation}\label{eq204}
	\mathcal{R}_{k+1} f_\nu \le \eta^2\epsilon^2.
	\end{equation}
	Then Proposition \ref{prop7} illustrates that
	$\mathscr{X}_{k,1/\nu}= {\tilde{\mathscr{V}}}_{k}\bar{\times}_{(N+1)} y_{k,\mu_k}$ satisfies
	\begin{equation*}
	\epsilon \le \left\| \mathcal{M}(\mathscr{X}_{k,1/\nu})-\mathscr{F}
	\right\|\le \eta \epsilon
	\end{equation*}
	If \eqref{eq204} does not holds, then we need to apply one more step of Algorithm \ref{a2} replacing $k$ with $k+1$.  As pointed out in \cite{Bentbib}, the bound \eqref{eq204} can be satisfied for small values of $k$.
	
\noindent	Assume that the bound \eqref{eq204} hold then we need to find the vector $y_{k,\mu_k}$
	by solving \eqref{eq33e} in which we set $\mu_k=1/\nu$ where $\nu$ satisfies \eqref{eq203}
	and \eqref{eq204}. Finally, we can determine the approximate solution
	$\mathscr{X}_{k,\mu_k}$
	by
	\[
	\mathscr{X}_{k,\mu_k}= {\tilde{\mathscr{V}}}_{k}\bar{\times}_{(N+1)} y_{k,\mu_k}.
	\]

\noindent Based on above discussions, we can construct two approaches based on GKB process to solve  \eqref{eq33}. These strategies are summarized in Algorithms \ref{a3} and \ref{a4}. In the next section, we numerically examine the feasibility of these algorithms. It turns out that each step of Algorithm \ref{a3} requires less CPU-time than Algorithm \ref{a4} to be performed.

\RestyleAlgo{ruled}
\LinesNumbered
\begin{algorithm}[htbp]
	\caption{GKB$_{-}$BTF-Tikhonov method with quadrature rules.\label{a3}}
	\textbf{Input:} Linear operator $\mathcal{M}$ ($\tilde{\mathcal{M}}$) and the right-hand side $\mathscr{F}$  ($\mathscr{D}$).\\
	Set user-chosen constant $\eta= 1.01$ and  $\varepsilon=\|\mathscr{E}\|$ where $ \mathscr{D}=\hat{\mathscr{D}}+\mathscr{E} $ in which $ \hat{\mathscr{D}} $ denote error-free tensor;\\
	Set $k=2$;\\
	Set $\mathscr{V}_1=\mathscr{F}/\|\mathscr{F}\|$;\\
	Determine the orthonormal bases $\{\mathscr{U}_j\}_{j=1}^{k+1}$ and $\{\mathscr{V}_j\}_{j=1}^k$ of tensors, and
	the bidiagonal matrices $T_k$ and $\bar{T}_k$ with Algorithm \ref{a2};\label{l4}\\
	Determine $\mu_k$ that satisfies \eqref{eq203} and \eqref{eq204} with Newton's method. This may require increasing  $k$, in this case set $k=k+1$ and go to Line \ref{l4}.\\
	Determine $y_{k,\mu_k}$ by solving \eqref{eq33e} and compute $\mathscr{X}_{k,\mu_k}= {\tilde{\mathscr{V}}}_{k}\bar{\times}_{(N+1)} y_{k,\mu_k}$;
\end{algorithm}

\RestyleAlgo{ruled}
\LinesNumbered
\begin{algorithm}[htbp]
	\caption{GK$_{-}$BTF-Tikhonov method (GKT$_{-}$BTF).\label{a4}}
	\textbf{Input:} Linear operator $\mathcal{M}$ ($\tilde{\mathcal{M}}$) and the right-hand side $\mathscr{F}$  ($\mathscr{D}$), $\varepsilon$, $\eta\geq 1$.\\
	Set $\mathscr{V}_1=\mathscr{F}/\|\mathscr{F}\|$;\\
	\For{k=1,2,\dots}{
		Determine the orthonormal bases $\{\mathscr{U}_j\}_{j=1}^{k+1}$ and $\{\mathscr{V}_j\}_{j=1}^k$ of tensors, and
		the bidiagonal matrices $T_k$ and $\bar{T}_k$ with Algorithm \ref{a2};\\
		Determine regularization parameter $\mu_k$ with by the discrepancy principle with user-chosen constant
		$\eta= 1.01$ \cite{Huang};\\
		Determine $y_{k,\mu_k}$  by solving a least-squares problem
		\[		\mathop {\min }\limits_{y \in {R^k}} {\left\| {\left( {\begin{array}{*{20}{c}}
					{{{\bar{ T }}_k}}\\
					{\sqrt \mu  I}
					\end{array}} \right)y - \left( {\begin{array}{*{20}{c}}
					{{\beta _1}}\\
					0
					\end{array}} \right)} \right\|_2}\]
		where $\beta_1=\|\mathscr{D}\| (\beta_1=\|\mathscr{F}\|)$\;
		\If{$\left\| \bar{T}_k y_{k,\mu_k} -\beta_1 e_1
			\right\|_2\leq \eta\varepsilon$,}
		{
		break;
	}
	}
	Compute $\mathscr{X}_{k,\mu_k}= {\tilde{\mathscr{V}}}_{k}\bar{\times}_{(N+1)} y_{k,\mu_k}$;\\
\end{algorithm}

\newpage
\section{Numerical experiments}\label{Sec4}
In this section, we report some numerical experiments to compare performances of the proposed methods. We limit ourselves to the case  $N=3$ in  \eqref{eq1} and \eqref{eql1}. In all  the test problems, the right-hand side tensors are assumed to be contaminated by an error tensor $\mathscr{E}$ which has normally distributed random entries with zero mean being scaled to have a specific level of noise
$\nu :={{\|\mathscr{E}\|}}/{{\|\mathscr{D}\|}}$ ($\nu :={{\|\mathscr{E}\|}}/{{\|\mathscr{F}\|}}$). All computations were carried out using Tensor Toolbox \cite{Bader} in \textsc{Matlab} R2018b with an Intel Core i7-4770K CPU @ 3.50GHz processor and 24GB RAM.

\noindent The relative error that we computed is given by
\[e_k:= \frac{{\Vert {{\mathscr{X}_{{\lambda _k},k}} - \hat{\mathscr{X}}} \Vert}}{{\Vert \hat{\mathscr{X}} \Vert}},\]
where $\hat{\mathscr{X}}$ denotes the desired solution of the error-free problem and $\mathscr{X}_{{\lambda _k},k}$ is the $k$-th computed approximation by the proposed algorithms.

\noindent In Tables \ref{tab5}, \ref{tab8}, \ref{tab2} and \ref{tab4}, the iterations were stopped when
\begin{equation}\label{eq71}
\left\| \mathcal{M}(\mathscr{X})-\mathscr{F}
\right\|\leq \eta \varepsilon,
\end{equation}
where $\eta$ is user-chosen constant and $\varepsilon$ is the norm of error, i.e.,  $\varepsilon=\|\mathscr{E}\|$. We comment that the norm in left-hand side of the above relation is computed inexpensively in view of  \eqref{eq201}.

\noindent For comparison with existing approaches in the literature, we use global Hessenberg process in conjunction with Tikhonov regularization based on tensor format (HT$_{-}$BTF) and flexible HT$_{-}$BTF (FHT$_{-}$BTF)  proposed in \cite{Najafi} for which we determine the regularization parameter by  {\it discrepancy principle} described in \cite{Hansen1}.
When the coefficient matrices are full, as anticipated, FHT$_{-}$BTF outperforms
other examined algorithms. However, for large and sparse coefficient matrices,
 FHT$_{-}$BTF needs more CPU time than Algorithms \ref{a3} and \ref{a4}.
Our observations illustrate that FHT$_{-}$BTF take a long time with respect to the stopping criterion \eqref{eq71} for large problems. Therefore, for the results reported in Tables \ref{tab8-1}, \ref{tab2-1} and \ref{tab4-1}, we used an alternative stopping criterion
given by,
\begin{equation}\label{eq72}
	\frac{{\| {{\mathscr{X}_{{ \lambda _k},k}} - {\mathscr{X}_{{\lambda _{k - 1}},k - 1}}} \|}}{{\| {{\mathscr{X}_{{\lambda _{k - 1}},k - 1}}} \|}} \le \tau,
	\end{equation}
where the maximum number of $40$ iterations was allowed. In FHT$_{-}$BTF method, we used two steps of stabilized biconjugate gradients based on tensor format (BiCGSTAB$_{-}$BTF) \cite{Chen} as the inner iteration; see \cite{Najafi} for further details.

\noindent We reported  the  required number of iterations and consumed CPU-time (in seconds) by algorithms to compute suitable approximate solutions satisfying the stopping criteria. For more clarification, we divide this section into two main parts. In Subsections \ref{Su2} and \ref{Su1}, we provide some numerical examples to solve ill-posed problems in the forms \eqref{eq1} and \eqref{eql1}, respectively.

\noindent To test the performance of algorithms for image restoration,
the exact solutions are tensors of sizes $ 576\times  787 \times 3 $\footnote{The corresponding color image is available at \url{https://www.hlevkin.com/TestImages/Boats.ppm}}  and $1019\times 1337\times 33$ which the second one associated with a hyperspectral image of natural scenes being also used in \cite[Example 5.3]{Najafi}.  Blurring matrices have the following forms in Subsections \ref{Su2} and \ref{Su1}, respectively,
\[
I \otimes I \times A^{(1)} + I \otimes A^{(2)} \otimes I + A^{(3)} \otimes I \otimes I,
\]
and
 \[ {I - {A^{(3)}} \otimes {A^{(2)}} \otimes {A^{(1)}}}, \]
 where $A^{(i)}$s are either Gaussian Toeplitz matrix $A=[a_{ij}]$ given by,
 \begin{align}\label{mat2}
 a_{ij} =&\begin{cases}
 \dfrac{1}{\sigma \sqrt{2\pi}}\exp\left(  -\dfrac{(i-j)^2}{2\sigma^2}\right) ,&|i-j|\leq r\\
 0,&\text{otherwise}
 \end{cases},
 \end{align}
 or the uniform Toeplitz matrix $B=[b_{ij}]$ defined by
 \begin{align}\label{mat1}
 b_{ij} =&\begin{cases}
 \dfrac{1}{2r-1},&|i-j|\leq r\\
 0,&\text{otherwise}
 \end{cases}.
 \end{align}
In literature, \eqref{mat2} and \eqref{mat1} have been used
as blurring matrices for testing applications of iterative schemes for image deblurring;  see \cite{Bentbib2016,Bentbib,Bouhamidi,Huang} for instance.
\subsection{Experimental results for ill-posed Sylvester tensor equations}\label{Su2}
As a first test problem, we consider \eqref{eq1} in which the coefficient matrices are full and extremely ill-conditioned. This kind of equations may arise from discretization of
a fully three-dimensional microscale dual phase lag problem by a mixed-collocation finite difference method; see \cite{Malek1,Malek2,Momeni-Masuleh} for further details.


\begin{example}\label{ex5}	Consider \eqref{eq1} with a perturbed right-hand side  such that  $A^{(\ell)}=[a_{ij}]$ for $\ell=1,2,3$ are defined by
		\[
		a_{ij}=
		\begin{cases}
		-2\left(\frac{\pi}{L} \right)^2\frac{(-1)^{i+j}}{\sin^2\left[\frac{1}{2}
			\left( \frac{2\pi \xi_j}{L}-x_i\right) \right]}, & i\ne j\\
		-\left(\frac{\pi}{L} \right)^2\left(\frac{n^2+2}{3} \right),& i=j,
		\end{cases}
		\]
		where  $ x_i=\frac{2\pi (i-1)}{n}, \,\xi_j=\frac{(j-1)L}{n},\, i,j=1,2,\dots,n $ with $L=300$.	The same problem  was solved by global schemes choosing odd values of $n$  for which the coefficient matrices $A^{(i)}$ are very well-conditioned; see \cite{Beik}. Similar to \cite[Example 5.4]{Najafi}, the value of $n$ is chosen to be even which results extremely ill-conditioned coefficient matrices. The error free right-hand side of \eqref{eq1} is constructed so that $\mathscr{X^\ast}=\mathrm{randn}(n,n,n)$ is its exact solution.
		The obtained numerical results  are disclosed in Table  \ref{tab5}.

	\begin{table}[htbp]
		\centering
		\caption{Comparison results for  Example \ref{ex5} with respect to stopping criterion \eqref{eq71}.}\resizebox{\textwidth}{!}{
		\begin{tabular}{ccclccc}
			\hline
			Grid  &$\mathrm{cond}(A^{(i)})$& Level of noise $(\nu)$ & Method &  Iter($ k $) & $e_k $   & CPU-times(sec) \\
			\hline
			\multirow{7}[7]{*}{$ 100\times100\times100$} & \multirow{7}[7]{*}{$1.25\cdot10^{16}$}&\multirow{3}[4]{*}{0.01} & Algorithm \ref{a3}     & 39&$1.11\cdot 10^{-1}$& 2.31\\
			&       &&  Algorithm \ref{a4}   &66  &$7.54\cdot10^{-2}$&3.34\\
			&    &   & HT$_{-}$BTF   &  11  &$7.51\cdot10^{-2}$&3.62\\
			&    &   & FHT$_{-}$BTF   &  5   &$6.25\cdot10^{-2}$& 0.98\bigstrut[b]\\
			\cline{3-7}      &    & \multirow{3}[4]{*}{0.001} &  Algorithm \ref{a3}   & 134    & $4.48\cdot10^{-2}$&7.53 \\
			&    &   &  Algorithm \ref{a4} & 126   &$4.49\cdot10^{-2}$&6.49 \\
			&    &   & HT$_{-}$BTF   &  24 &$2.57\cdot10^{-2}$& 34.84 \\
			&    &   & FHT$_{-}$BTF   &    8 &$2.33\cdot10^{-2}$&2.35 \bigstrut[b]\\
			\hline
			\multirow{7}[7]{*}{$ 150\times 150\times 150$} &\multirow{7}[7]{*}{$4.67\cdot10^{16}$}& \multirow{3}[4]{*}{0.01} &  Algorithm \ref{a3}  & 37    & $1.18\cdot10^{-1}$ &7.03 \\
			&     &  &Algorithm \ref{a4}     & 103  &$5.88\cdot10^{-2}$&17.28\\
			&    &   & HT$_{-}$BTF   & 11   & $7.40\cdot10^{-2}$&12.32  \\
			&    &   & FHT$_{-}$BTF   &   5  &$6.33\cdot10^{-2}$&3.34 \\
			\cline{3-7}     &     & \multirow{3}[4]{*}{0.001} & Algorithm \ref{a3}   & 178    & $4.02\cdot10^{-2}$ &36.89\bigstrut[t]\\
			&   &    & Algorithm \ref{a4}    & 193    &$3.30\cdot10^{-2}$ &36.97\\
			&    &   & HT$_{-}$BTF   &   21 & $3.21\cdot10^{-2}$&72.95 \\
			&    &   & FHT$_{-}$BTF   &  8   &$2.61\cdot10^{-2}$& 8.15\bigstrut[b]\\
			\hline
			\multirow{7}[7]{*}{$ 180\times180\times180$}& \multirow{7}[7]{*}{$3.28\cdot10^{16}$}& \multirow{3}[4]{*}{0.01} & Algorithm \ref{a3}& 36    & $1.19\cdot10^{-1}$&11.13\bigstrut[t]\\
			&&       &     Algorithm \ref{a4} &   127   & $5.38\cdot10^{-2}$&35.85\\
			&    &   & HT$_{-}$BTF   & 11   & $7.55\cdot10^{-2}$&21.45  \\
			&    &   & FHT$_{-}$BTF   & 5    &$6.13\cdot10^{-2}$&5.66 \bigstrut[b]\\
			\cline{3-7}    &      & \multirow{3}[4]{*}{0.001} & Algorithm \ref{a3}     & 154    & $4.18\cdot10^{-2}$& 58.64\bigstrut[t]\\
			&   &    &   Algorithm \ref{a4} &  231    &$2.88\cdot10^{-2}$ & 73.47\\
			&    &   & HT$_{-}$BTF   &   22 & $2.89\cdot10^{-2}$&  134.51\\
			&    &   & FHT$_{-}$BTF   &  7   &$2.91\cdot10^{-2}$& 11.65\\
			\hline
		\end{tabular}%
	}
		\label{tab5}%
	\end{table}%

	
\end{example}		
As can be  seen in Table \ref{tab5}, FHT$_{-}$BTF works better than the other approaches, this could be expected as the coefficient matrices are full. Now we present experimental results related to image restoration. In fact, error free right-hand sides in \eqref{eq1} is constructed such that the exact solution is a hyperspectral image. Here the matrices $A^{(i)}$s ($i=1,2,3$) are sparse and it is observed that Algorithm \ref{a3} surpasses other examined iterative schemes. \\

\begin{example}\label{ex8}
	We consider the case that a tensor of order $ 1019\times 1337\times33 $ is the exact solution of \eqref{eq1} which corresponds to a hyperspectral image of natural scenes\footnote{\url{http://personalpages.manchester.ac.uk/staff/d.h.foster}}.The coefficient matrices $ A^{(1)},A^{(2)} $ and $ A^{(3)} $ are given by \eqref{mat1} with suitable dimensions such that $ r = 2 $  for $ A^{(1)}$, $A^{(2)} $  and $ r = 3$  for $ A^{(3)}$  which result $ \mathrm{cond}(A^{(1)})=5.26\cdot10^{16}$, $\mathrm{cond}(A^{(2)})=  1.75\cdot10^{17}$ and $\mathrm{cond}(A^{(3)})=4.75\cdot10^{16}$.

The obtained numerical results are disclosed in Table  \ref{tab8} for which the algorithms was terminated once \eqref{eq71} satisfied.
As pointed out earlier, (F)HT$_{-}$BTF method can not be efficiently used with respect to stopping criterion \eqref{eq71}. Therefore, we rerun all of the algorithms
with respect to \eqref{eq72} and report the results in Table \ref{tab8-1}.
As seen, Algorithms \ref{a3} and \ref{a4} work better than  (F)HT$_{-}$BTF.
We further comment that Algorithm \ref{a3} consumes less CPU-time than Algorithm \ref{a4}.
	\begin{table}[htbp]
		\centering
		\caption{Results for  Example \ref{ex8} with respect to stopping criterion \eqref{eq71}.}
		\begin{tabular}{clccc}
			\hline
			Level of noise $(\nu)$  & \multicolumn{1}{c}{Method} &  Iter($ k $)     &$e_k $   & CPU-times(sec) \\
			\hline
			\multirow{2}[2]{*}{0.01}
			&  Algorithm \ref{a3}    & 6  & $3.54\cdot10^{-2}$ &14.31 \\
			& Algorithm \ref{a4}      & 6&  $3.54\cdot10^{-2}$&19.57 \\
			\hline
			\multirow{2}[2]{*}{0.001}
			& Algorithm \ref{a3}     &20    & $1.72\cdot10^{-2}$& 57.03\\
			& Algorithm \ref{a4}    &     20 & $1.72\cdot10^{-2}$ & 65.44	\\
			\hline
		\end{tabular}%
		\label{tab8}%
	\end{table}%
	\begin{table}[htbp]
		\centering
		\caption{Results for  Example \ref{ex8}  with respect to stopping criterion \eqref{eq72} using $ \tau=2\cdot10^{-2}$.}
		\begin{tabular}{clccc}
			\hline
			Level of noise $(\nu)$  & \multicolumn{1}{c}{Method} &  Iter($ k $)     &$e_k $   & CPU-times(sec) \bigstrut\\
			\hline
			\multirow{3}[3]{*}{0.01}
			&  Algorithm \ref{a3}    &    4   &$ 3.85\cdot10^{-2}$&7.45 \bigstrut[t]\\
			& Algorithm \ref{a4}      &  4 &  $3.88 \cdot10^{-2}$&  14.39 \\
			& HT$_{-}$BTF     &4 &  $6.77\cdot10^{-2}$&   27.01\\
			& FHT$_{-}$BTF     & 2 &  $6.28\cdot10^{-2}$&  26.81	\\
			\hline
			\multirow{3}[3]{*}{0.001}
			& Algorithm \ref{a3}     &    4    &$ 3.57\cdot10^{-2}$&7.47\\
			& Algorithm \ref{a4}     &      4  &  $3.65 \cdot10^{-2}$& 14.49\\
			& HT$_{-}$BTF     &     4 &  $4.40 \cdot10^{-2}$& 26.67\\
			& FHT$_{-}$BTF     &2  &  $3.85\cdot10^{-2}$&   25.30	\\
			\hline
		\end{tabular}%
		\label{tab8-1}%
	\end{table}%
	
\end{example}

\subsection{Experimental results for ill-posed Stein tensor equations}\label{Su1}

In  this subsection, we apply the proposed approaches for solving two ill-posed problems in the form \eqref{eql1}. Here, error free right-hand sides are constructed such that exact solutions of \eqref{eql1} are color images.
The iterations in the algorithms were stopped in two different ways, i.e., \eqref{eq71} and  \eqref{eq72} are used separately.

\begin{example}\label{ex2}
	This example is concerned with the restoration of a color image. The “original” exact image\footnote{The image is available at \url{https://www.hlevkin.com/TestImages/Boats.ppm}}  is stored by a $ 576\times  787 \times 3 $ tensor.
	We consider \eqref{eql1} in which  $A^{(1)}$ is given by \eqref{mat2}, $A^{(2)}$ and $ A^{(3)}$ are given by \eqref{mat1} with suitable dimensions. Here we set  $r=7,\sigma=2$ for $A^{(1)}$ and $r=2$ for $A^{(2)}$ and $A^{(3)}$. It can be seen that
	$ \mathrm{cond}(A^{(1)})=1.79\cdot10^{6}$ and $
	\mathrm{cond}(A^{(2)})=   4.05\cdot10^{17}$ and $
	\mathrm{cond}(A^{(3)})=   6.45\cdot10^{49}$.
	\begin{table}[htbp]
		\centering
		\caption{Results for  Example \ref{ex2} with respect to stopping criterion \eqref{eq71}.}
			\begin{tabular}{clccc}
				\hline
				Level of noise $(\nu)$  & \multicolumn{1}{c}{Method} &  Iter($ k $)     &$e_k $   & CPU-times(sec) \bigstrut\\
				\hline
				\multirow{4}[4]{*}{0.01}
				&  Algorithm \ref{a3}    & 13    & $5.31\cdot10^{-2}$ &0.08\bigstrut[t]\\
				& Algorithm \ref{a4}      &  13    &  $5.32\cdot10^{-2}$&   1.03\\
				& HT$_{-}$BTF     &   12    & $6.46\cdot10^{-2}$& 8.71\\
				& FHT$_{-}$BTF     &   5  & $6.58\cdot10^{-2}$&2.25	\bigstrut[b]\\
				
				\hline
				\multirow{4}[4]{*}{0.001}
				& Algorithm \ref{a3}     & 52     &  $2.63\cdot10^{-2}$&  3.78\bigstrut[t]\\
				& Algorithm \ref{a4}     &     63  &  $2.46\cdot10^{-2}$& 5.62 \\
				& HT$_{-}$BTF     &     13   & $5.97\cdot10^{-2}$&10.64\bigstrut[b]\\
				& FHT$_{-}$BTF     &  6   & $6.41\cdot10^{-2}$&3.02\bigstrut[b]\\
				\hline
			\end{tabular}%
		\label{tab2}%
	\end{table}%
	
	\begin{table}[htbp]
		\centering
		\caption{Results for  Example \ref{ex2} with respect to stopping criterion \eqref{eq72} using $ \tau=2\cdot10^{-2} $.}
		\begin{tabular}{clccc}
			\hline
			Level of noise $(\nu)$  & \multicolumn{1}{c}{Method} &  Iter($ k $)     &$e_k $   & CPU-times(sec) \bigstrut\\
			\hline
			\multirow{4}[4]{*}{0.01}
			&  Algorithm \ref{a3}    & 6      & $7.40\cdot10^{-2}$ &0.41\bigstrut[t]\\
			& Algorithm \ref{a4}     &   6     &  $7.37\cdot10^{-2}$&   0.49\\
			& HT$_{-}$BTF     &     9   & $7.88\cdot10^{-2}$&4.09\\
			& FHT$_{-}$BTF     &   6    & $6.54\cdot10^{-2}$&2.86	\bigstrut[b]\\
			
			\hline
			\multirow{4}[4]{*}{0.001}
			& Algorithm \ref{a3}     &     6  &  $7.32\cdot10^{-2}$&0.38\bigstrut[t]\\
			& Algorithm \ref{a4}     &     6     &  $7.34\cdot10^{-2}$&0.50\\
			& HT$_{-}$BTF     &     9    & $7.84\cdot10^{-2}$&4.14\bigstrut[b]\\
			& FHT$_{-}$BTF     &   6    & $6.42\cdot10^{-2}$&2.87\bigstrut[b]\\
			\hline
		\end{tabular}%
		\label{tab2-1}%
	\end{table}%

\end{example}

\medskip
\begin{example}\label{ex4}
	We consider the case that a tensor of order $ 1019\times 1337\times33 $ is the exact solution of \eqref{eql1}. The coefficient matrices $ A^{(1)},A^{(2)} $ and $ A^{(3)} $ are  defined by \eqref{mat1} with suitable dimensions such that $ r = 12$  for $ A^{(1)}$ and $ r = 2$  for $A^{(2)} $  and $ r = 6$  for $ A^{(3)}$. Here, we have $ \mathrm{cond}(A^{(1)})=2.05\cdot10^{18}$, $\mathrm{cond}(A^{(2)})=  1.75\cdot10^{17}$ and $\mathrm{cond}(A^{(3)})=2.44\cdot10^{17}$.
\begin{table}[htbp]
	\centering
	\caption{Results for  Example \ref{ex4} with respect to stopping criterion \eqref{eq71}.}
	\begin{tabular}{clccc}
		\hline
		Level of noise $(\nu)$  & \multicolumn{1}{c}{Method} &  Iter($ k $)     &$e_k $   & CPU-times(sec) \bigstrut\\
		\hline
		\multirow{2}[2]{*}{0.01}
		&  Algorithm \ref{a3}    & 18    &$7.98\cdot10^{-2}$&58.76\bigstrut[t]\\
		& Algorithm \ref{a4}    & 18   &  $7.97\cdot10^{-2}$& 70.09 	\bigstrut[b]\\
		\hline
		\multirow{2}[2]{*}{0.001}
		& Algorithm \ref{a3}     & 31     &$5.62\cdot10^{-2}$&96.14\bigstrut[t]\\
		& Algorithm \ref{a4}     &     35  &  $5.22\cdot10^{-2}$&695.76\bigstrut[b]\\
		\hline
	\end{tabular}%
	\label{tab4}%
\end{table}%
\begin{table}[htbp]
	\centering
	\caption{Results for  Example \ref{ex4} with respect to stopping criterion \eqref{eq72} using $ \tau=3\cdot10^{-2} $.}
	\begin{tabular}{clccc}
		\hline
		Level of noise $(\nu)$  & \multicolumn{1}{c}{Method} &  Iter($ k $)     &$e_k $   & CPU-times(sec) \bigstrut\\
		\hline
		\multirow{3}[3]{*}{0.01}
		&  Algorithm \ref{a3}    &   6   &$ 1.39\cdot10^{-1}$&17.73 \bigstrut[t]\\
		& Algorithm \ref{a4}      &    6 &  $1.38 \cdot10^{-1}$&  26.31 \\
		& HT$_{-}$BTF     &  6   &  $2.61\cdot10^{-1}$& 54.83\\
		& FHT$_{-}$BTF     &    4 &  $1.40\cdot10^{-1}$&   57.04	\bigstrut[b]\\
		\hline
		\multirow{3}[3]{*}{0.001}
		& Algorithm \ref{a3}     &    6   &$ 1.37\cdot10^{-1}$&17.74\bigstrut[t]\\
		& Algorithm \ref{a4}    &     6   &  $1.38 \cdot10^{-1}$& 26.52\\
		& HT$_{-}$BTF     &  6    &  $2.44\cdot10^{-1}$& 54.94   \\
		& FHT$_{-}$BTF     &   4   &  $1.38\cdot10^{-1}$&   57.96	\bigstrut[b]\\
		\hline
	\end{tabular}%
	\label{tab4-1}%
\end{table}%

%
\end{example}
The obtained numerical results for Examples \ref{ex2} and \ref{ex4} are disclosed in Tables  \ref{tab2}, \ref{tab2-1}, \ref{tab4} and \ref{tab4-1}.
Similar to what we observed for second example of previous subsection, Algorithm \ref{a3} is superior to other examined approaches.


\section{Conclusions}\label{Sec5}
In this paper, we first present some results for conditioning of the Stein tensor equation.
Then, we proposed the global Golub--Kahan bidiagonalization process with applications for solving ill-posed linear tensor equations such as Sylvester and Stein tensor equations where the iterative schemes can be also implemented for an arbitrary linear operator over $\mathbb{R}^{n_1\times n_2\times\cdots \times n_k}$. We gave some new theoretical results and present  some numerical examples with applications to color image restoration to show  the applicability and the effectiveness of the  proposed schemes for computing solutions of high quality.


\begin{thebibliography}{99}
	
	
	\bibitem{Ballani} Ballani J and Grasedyck L.
	A projection method to solve linear systems in tensor format.
	\textit{Numerical Linear Algebra with Applications.} 2013; \textbf{20}(1): 27--43.
	
	\bibitem{Bader}
	Bader BW and Kolda TG. MATLAB Tensor Toolbox Version 2.5.
	\url{http://www.sandia.gov/~tgkolda/TensorToolbox}.
	
	
	\bibitem{Beik} Beik FPA, Movahed FS and  Ahmadi-Asl S.
	On the Krylov subspace methods based on tensor format for
	positive definite Sylvester tensor equations.
	\textit{Numerical Linear Algebra with Applications.} 2016; \textbf{23}(3): 444--466.
	
	\bibitem{Bentbib2016}
	Bentbib AH, Guide M. El, Jbilou K and Reichel L.
	A global Lanczos method for image restoration,
	\textit{Journal of Computational and Applied Mathematics.} 2016; \textbf{300} 233--244.
	
	
	\bibitem{Bentbib}
	Bentbib AH, Guide M. El, Jbilou K and Reichel L.
	Global Golub–Kahan bidiagonalization applied to large discrete ill-posed problems.
	\textit{Journal of Computational and Applied Mathematics.} 2017; \textbf{322} 46--56.
	
	\bibitem{Bouhamidi}
	Bouhamidi A,  Jbilou K, Reichel L, and Sadok H.
	A generalized global Arnoldi method for
	ill-posed matrix equations.
\textit{Journal of Computational and Applied Mathematics.} 2012; \textbf{236}  2078--2089.
	
	\bibitem{Buccini}
	Buccini A.
	Tikhonov--type iterative regularization methods for ill-posed inverse problems:theoretical aspects and applications. PhD thesis. University of Insubria.\\
	{\tiny{\url{http://insubriaspace.cineca.it/bitstream/10277/703/1/Phd_Thesis_Buccinialessandro_completa.pdf}}}
	
	\bibitem{Chen}
	Chen Z and Lu LZ.
	A projection method and Kronecker product preconditioner for solving Sylvester tensor equations.
	\textit{Science China Mathematics.} 2012; \textbf{55}(6): 1281--1292.
	
	\bibitem{Cichocki} {{Cichocki A, Zdunek R, Phan AH, Amari SI. \textit{Nonnegative matrix and tensor factorizations: applications to exploratory multi-way data analysis and blind source separation.} John Wiley \& Sons, 2009.}}
	
	
	\bibitem{Golubbook} Golub GH  and Van Voan CF,
	\textit{Matrix Computations.} The Johns Hopkins University Press,  Batimore. 1996.
	
	
	\bibitem{Horn}
	Horn RA and Johnson CR.
	\textit{Matrix Analysis.}
	Cambridge University Press, Cambridge, UK, 1985.
	
	
	\bibitem{Huang}
	Huang G, Reichel L, and Yin F.
	On the choice of subspace for large-scale Tikhonov
	regularization problems in general form.
		\textit{ Numerical Algorithms.} in press.
	\url{DOI: 10.1007/s11075-018-0534-y}.
	
	
		\bibitem{Huang2019}
		Huang B, Xie Y, and Ma C.
		Krylov subspace methods to solve a class of tensor equations
		via the Einstein product
		\textit{Numerical Linear Algebra with Applications.} 2019; in press.
		\url{DOI: 10.1002/nla.2254}.
	
	\bibitem{Kolda}
	Kolda TG and Bader BW.
	Tensor decompositions and applications.
	\textit{SIAM Review.} 2009; \textbf{51}(3): 455--500.
	
	
	\bibitem{Kressner} Kressner D and Tobler C.
	Low-rank tensor Krylov subspace methods for parametrized linear systems.
	\textit{SIAM Journal on Matrix Analysis and Applications.} 2011; \textbf{32}(1):1288--1316.
	
	\bibitem{Liang}
	Liang L and Zheng B.
	Sensitivity analysis of the Lyapunov tensor equation.
	\textit{Linear Multilinear Algebra.} 2019; \textbf{67}(3):555--572.
	
	
	\bibitem{Malek1} Malek A, Bojdi ZK, Golbarg PNN.
	Solving fully three-dimensional microscale dual phase lag problem using mixed-collocation finite difference discretization.
	\textit{Journal of Heat Transfer.} 2012; \textbf{134}(9),  094504.
	
	\bibitem{Malek2} Malek A, Masuleh SHM.
	Mixed collocation--finite difference method for 3D microscopic
	heat transport problems.
	\textit{Journal of Computational and Applied Mathematics.} 2008; \textbf{217}(1):137--147.
	
	\bibitem{Momeni-Masuleh}
	Masuleh SHM, Phillips TN.
	Viscoelastic flow in an undulating tube using spectral methods.
	\textit{Computers} \& \textit{Fluids.} 2004; \textbf{33}(8):1075--1095.
	
	
	\bibitem{Song}
	Sun YS, Jing M and Li BW.
	Chebyshev collocation spectral method for three-dimensional transient coupled radiative--conductive Heat transfer. \textit{Journal of Heat Transfer.} 2012; \textbf{134}: 092701--092707.
	
	
	
	\bibitem{Najafi} Najafi-Kalyani M, Beik FPA and Jbilou K.
	On global iterative schemes based on Hessenberg process for (ill-posed) Sylvester tensor equations,
	\textit{Journal of Computational and Applied Mathematics.} DOI:  10.1016/j.cam.2019.03.045.
	
	
	\bibitem{Xu} Xu X and Wang Q-.W.
	Extending BiCG and BiCR methods to solve the Stein tensor equation.
	\textit{Computers \& Mathematics with Applications.} 2019; \textbf{77}(12): 3117--3127.
	
	\bibitem{Zak}
	Zak MK and Toutounian F.
	Nested splitting conjugate gradient method for matrix
	equation $AXB = C$ and preconditioning,
	\textit{Computers \& Mathematics with Applications.}  2013;
	\textbf{66} 269--278.
	
	
\bibitem{Hansen1}	Hansen PC. Regularization tools: A Matlab package for analysis and solution of discrete ill-posed problems.\textit{Numerical Algorithms.} 1994; \textbf{6}, 1--35. Software
is available in \url{http://www.netlib.org}
\end{thebibliography}
\end{document}